\documentclass[reqno]{amsart}

    \usepackage{latexsym,amsfonts, amssymb,amsmath,amscd,bm,amscd,amsthm, mathrsfs,
        amsxtra,epsfig,setspace,appendix,psfrag}

\usepackage{color}
\usepackage{cleveref}
\usepackage{dsfont}
\usepackage{breqn}
\usepackage{xcolor}
\usepackage{relsize}

\newtheorem{thm}{Theorem}
\newtheorem{defn}{Definition}
\newtheorem{prop}{Proposition}
\newtheorem{lem}{Lemma}

\newcommand{\perm}{\mathrm{perm}}
\newcommand{\adj}{\mathrm{adj}}

\newcommand{\abs}[1]{\left\vert#1\right\vert}

\newcounter{example}[section]
\newenvironment{example}[1][]{\refstepcounter{example}\par\medskip
   \noindent \textbf{Example~\theexample. #1} \rmfamily}{\medskip}

\usepackage[reqno]{amsmath}

\begin{document}
\title{Bergman Kernels of Two Dimensional Monomial Polyhedra} 
\author{Rasha Almughrabi}
\begin{abstract}
We compute explicitly the Bergman kernels of all two dimensional monomial polyhedra, a class of domains including the Hartogs triangle and some of its generalizations. The kernel is computed from the representation of such domains as quotient of the product of a disc and punctured disc.
\end{abstract}
\address{Department of Mathematics\\Central Michigan University\\Mt.Pleasant,MI 48859, USA}
\email{almug1ra@cmich.edu}
\maketitle

 \section{Introduction}
The Bergman kernel is a fundamental object of interest in modern complex analysis. In general, it is not possible to compute the Bergman kernel of a domain in explicit closed form. The major exceptions are the ball and other bounded symmetric domains \cite{hua}.

Apart from this, it is also possible in principle to use transformation laws of the Bergman kernel to compute the Bergman kernels of domains which arise as quotients of domains for which the Bergman kernel is known. However, such computations lead to complicated algebraic and computational issues.

In this paper, we will explicitly compute the Bergman kernel  of domains belonging to a class called \emph{monomial polyhedra} which arise as quotients of product domains where the factors are discs or punctured discs. Although this problem is interesting in $\mathbb{C}^n$, here we give a full computation for the case $n=2$. A monomial polyhedron can be thought as a generalization of the famous \emph{Hartogs Triangle} $\mathcal{H}(1,-1)= \{ z \in \mathbb{C}^2: \abs{z_1} < \abs{z_2} <1 \}$. They are of current interest in complex analysis, see \cite{bender} and \cite{edho}.

We begin by describing the class of domains that we study. By a \emph{domain} we mean an open connected set in $\mathbb{C}^n$, for $n \geq 1$.
  
\begin{defn}
Let $\mathcal{U} \subset \mathbb{C}^n$ be a bounded domain. We say that $\mathcal{U}$ is a \emph{monomial polyhedron} if there is an integer matrix $B \in M_n(\mathbb{Z})$ such that
\begin{equation}\label{dom1}
\mathcal{U} = \biggl\{ z \in \mathbb{C}^n: \abs{z_1}^{b_1^j}\cdots \abs{z_n}^{b_n^j} <1, 1\leq j \leq n \biggr\},
\end{equation}
where $b_j^i$ stands for the entry of $B$ in the $i$-th row and $j$-th column.
\end{defn}

If $\Omega \subset \mathbb{C}^n$ is a monomial polyhedron then it is a \emph{Reinhardt domain}, that is, if $z \in \Omega$ and $\theta_j \in \mathbb{R}$ for $j=1,\cdots,n$, then we have
\begin{equation}\nonumber
\left( e^{i \theta_1 }z_1, \cdots, e^{i \theta_n }z_n \right) \in \Omega.
\end{equation}

If $\Omega$ is a Reinhardt domain in $\mathbb{C}^n$, its \emph{shadow} is the subset $\abs{\Omega}$ of $\mathbb{R}^n$ given by
\begin{equation}\nonumber
\abs{\Omega} = \{ \left( \abs{z_1}, \cdots, \abs{z_n} \right): z \in \Omega \}.
\end{equation}
In fact, a monomial polyhedron is a so-called a \emph{pseudoconvex domain}, the type of domain most important in complex analysis of several variables. A monomial polyhedron typically has a singularity at origin, see \cite{range}.

Many special examples of computations of Bergman kernel of monomial polyhera are found in the literature.
Bremermann in \cite{bremer} calculated the Bergman kernel for the classical Hartogs triangle. Edholm in \cite{luke1} (see also \cite{edho}) explicitly computed the Bergman kernel for the thin and fat Hartogs triangle. In \cite{austin}, Chakrabarti, Konkel, Mainkar, and Miller give an explicit form for the Bergman kernels of elementary Reinhardt domains. In \cite{chen}, Chen gives a formula for the Bergman kernel for the $n$-dimensional generalized Hartogs triangle. Chen applies the transformation formula to obtain the Bergman kernel of the $n$-dimensional generalized Hartogs triangle. Later in \cite{zhang}, Zhang explicitly calculated the Bergman kernel of the $n$-dimensional generalized Hartogs triangles.

\subsection{Two Examples}
In this section we give two examples of applications of the general formula found in this paper. The Bergman kernels of these examples do not seem to have been noted in the literature before.
\begin{example} \label{example8}
Consider the domain in $\mathbb{C}^2$ given as

\begin{equation}\nonumber
\mathcal{U} = \{ z \in \mathbb{C}^2 : \abs{z_1}^4 < \abs{z_2} < \abs{z_1 } ^{\frac{1}{3} } \}.
\end{equation}
The associated matrix of the domain in $\mathcal{U}$ is given by $B=\left[ \begin{array}{cc}
4 & -1 \\  \nonumber
-1 & 3
\end{array}
\right],$ and the Reinhardt shadow of $\mathcal{U}$ shown below in Figure \ref{here05}.

The Bergman kernel of $\mathcal{U}$, by using Theorem \ref{mainthm1} and Theorem \ref{mainthm2} below, is given by
\begin{equation}\nonumber
\begin{aligned}
K_{\mathcal{U}}(z,w)=  \dfrac{1}{11 \pi^2 }\cdot \dfrac{g(t_1, t_2)}{ \left(t_2-t_1^{4}\right)^2 \left(t_1-t_2^{3} \right)^2},
\end{aligned}
\end{equation}
where
\begin{eqnarray}\nonumber
g(t_1,t_2) &= 7t_2^5 + 6 t_2^6 + 4t_1 t_2^2 + 16 t_1 t_2^3+30 t_1 t_2^4 + 24 t_1 t_2^5 + 10 t_1 t_2^6 + 3 t_1^2 t_2 + 20 t_1^2 t_2^2 + 45 t_1 ^2 t_2^3+ 54 t_1^2 t_2^4 \\ \nonumber
&+ 35 t_1^2 t_2^5 + 8 t_1^2 t_2^6 + 12 t_1^3 t_2 + 42 t_1^3 t_2^2 + 80 t_1^3 t_2^3 + 72 t_1^3 t_2^4 + 40 t_1^3 t_2^5 + 27 t_1^4 t_2 + 70 t_1^4 t_2^2 + 121 t_1^4 t_2^3 \\ \nonumber
&+ 70t_1^4 t_2^4 + 27 t_1^4 t_2^5 + 40 t_1^5 t_2 +72 t_1^5 t_2^2 + 80 t_1^5 t_2^3 + 42 t_1^5 t_2^4 + 12 t_1^5 t_2^5 + 8t_1^6 + 35 t_1^6 t_2 + 54 t_1^6 t_2^2 + 45 t_1^6 t_2^3 \\  \nonumber
&+ 20 t_1^6 t_2^4 + 3 t_1^6 t_2^5 + 10 t_1^7 + 24 t_1^7 t_2 + 30 t_1^7 t_2^2 + 16 t_1^7 t_2^3 +4 t_1^7 t_2^4 + 6t_1^8 + 7 t_1^8 t_2.
\end{eqnarray}

\end{example}

\begin{example} \label{example9}
Consider the domain in $\mathbb{C}^2$ given as
\begin{equation}\nonumber
\mathcal{V} = \{ z \in \mathbb{C}^2 : \abs{z_1}^{1/2} < \abs{z_2} < \abs{z_1 } ^{1/4 } \}.
\end{equation}
The associated matrix of the domain in $\mathcal{V}$ is given by $B=\left[ \begin{array}{cc}
1 & -2 \\  \nonumber
-1 & 4
\end{array}
\right],$ and the Reinhardt shadow of $\mathcal{V}$ shown in Figure \ref{here05}.

The Bergman kernel of $\mathcal{V}$, by using Theorem \ref{mainthm1} and Theorem \ref{mainthm2} below, is given by
\begin{eqnarray}\nonumber
K_{\mathcal{U}}(p,q) = \dfrac{t_2^8 + t_1 t_2^4 + 4 t_1 t_2^5 + t_1 t_2^6 +t_1^2 t_2^2}{2 \pi^2  \left(t_2^2-t_1 \right)^2   \left(t_1-t_2^4 \right)^2}.
\end{eqnarray}
\begin{figure}[h]
\centering
\includegraphics[scale=0.4]{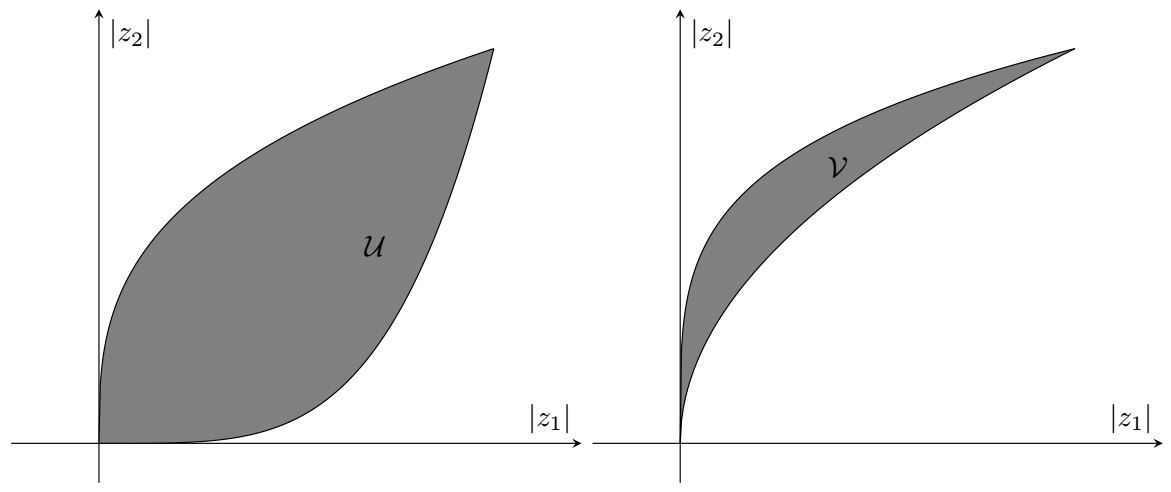}
\caption{The Reinhardt shadow of $\mathcal{U}$ (left) and $\mathcal{V}$ (right). }
\label{here05}
\end{figure}
\end{example}

\subsection{Main Results}\label{fundis}
Let $\displaystyle B=\left[ \begin{array}{cc}
b_1^1 & b_2^1\\  \nonumber
b_1^2 & b_2^2
\end{array}
\right] \in M_2(\mathbb{Z})$ be a $2 \times 2$ integer matrix such that the domain 
\begin{equation}\label{defdef}
\mathcal{U}= \left\{ (z_1,z_2) \in \mathbb{C}^2: |z_1|^{b_1^1} |z_2|^{b_2^1}<1,\: |z_1|^{b_1^2} |z_2|^{b_2^2}<1 \right\}
\end{equation}
is bounded, and therefore a monomial polyhedron in $\mathbb{C}^2$. Define the matrix $A \in M_2(\mathbb{Z})$ by
\begin{equation}\label{matAAA}
A=\left[ \begin{array}{cc}
a_1^1 & a_2^1\\  
a_1^2 & a_2^2
\end{array}
\right]=\left[ \begin{array}{cc}
\dfrac{b_2^2}{\gcd(b_2^2,b_1^2)} & \dfrac{-b_2^1}{\gcd(b_2^1,b_1^1)}\\ 
\dfrac{-b_1^2}{\gcd(b_2^2,b_1^2)} & \dfrac{b_1^1}{\gcd(b_2^1,b_1^1)}
\end{array}
\right].
\end{equation}
It is not difficult to see that \eqref{defdef} represents a bounded domain only if $\det(B)\ne0$. Without loss of generality, assume that $\det(B)>0$, if not, we can interchange the two rows of $B$ and this will not change the domain $\mathcal{U}$.
Our first main result is
\begin{thm}\label{mainthm1}
The Bergman kernel of domain $\mathcal{U}$ defined in \eqref{defdef} is of the form
\begin{equation}\label{summation5}
K_{\mathcal{U}}(z,w)=  \dfrac{1}{\pi^2 \det(A)}\cdot \dfrac{g(t_1,t_2)}{ \left(t_2^{a_2^1}-t_1^{a_2^2}\right)^2 \left(t_1^{a_1^2}-t_2^{a_1^1} \right)^2} ,
\end{equation}
where $z,w\in \mathcal{U}$, $t_j=z_j \overline{w_j}$, $j=1,2$, and $g(t_1,t_2) \in \mathbb{Z}[t_1,t_2]$ is a polynomial with integer coefficients, which will be described in Theorem \ref{mainthm2} below.
\end{thm}
In order to describe the polynomial $g$ in Theorem \ref{mainthm1}, we introduce the following function. For integers $k\geq 1$ and $r$, let
\begin{equation}\label{dkr}
D_{k}(r)=
\left\{ 
\begin{array}{ll}
1+r, & 0\leq r \leq k-1,\\
2k-(1+r),& k \leq r \leq 2k-2,\\
0, & r <0 \quad \text{or} \quad r >2k-2.
\end{array}
\right.
\end{equation}
The next lemma, whose proof is clear, explains why we define $D_k$ in this way.
\begin{lem}\label{lemma1}
For a positive integer $k$, we have
\begin{equation}\label{event}
\left(\frac{1-x^k}{1-x}\right)^2=\mathlarger{\sum_{r =0}^{2k-2}}D_k(r) x^r,
\end{equation}
where $D_k$ is as in \eqref{dkr}.
\end{lem}
Beside $D_k$, we need the following notations.
\begin{enumerate}
\item The permanent of the matrix $P=\left[ \begin{array}{cc}
p_1^1 & p_2^1\\  \nonumber
p_1^2 & p_2^2
\end{array}
\right]$ is denoted by $\perm(P)$.

Recall that $\perm(P)= p_1^1 p_2^2 +p_2^1 p_1^2.$
\item  We set for a matrix $P$ and $\gamma=(\gamma_1, \gamma_2) \in \mathbb{N}^2$, 
\begin{equation}\label{funnn}
N(P,\gamma)= p_1^1 \gamma_1 + p_1^2 \gamma_2 -2p_1^1 p_1^2 + p_1^1+ p_1^2-1 -\perm(P) +\abs{\det(P)}.
\end{equation}
(In this paper, $\mathbb{N}$ is the set of positive integers).
\item Let $R$ denote the matrix
\begin{equation}
R=\left[ \begin{array}{cc}
0 & 1\\  \nonumber
1 & 0
\end{array}
\right].
\end{equation}
\end{enumerate}
\begin{thm}\label{mainthm2}
In formula \eqref{summation5} we have, with above notations

\begin{equation}\nonumber
g(t_1,t_2) = \mathlarger{\sum\limits_{\gamma \in \mathbb{N}^2}}D_{\det(A)}(N(A,\gamma))D_{\det(A)}(N(AR,\gamma))t_1^{\gamma_1} t_2^{\gamma_2},
\end{equation}

where the sum is finite, in fact, the coefficients of $g$ can be nonzero only if
\begin{eqnarray}\label{range1}
\begin{aligned}
0 \leq \gamma_1 \leq 2a_1^2 +2a_2^2 -2, \\
0 \leq \gamma_2 \leq 2 a_1^1 + 2 a_2^1  - 2.
\end{aligned}
\end{eqnarray}
 
\end{thm}
In fact, one can be more precise about the set of indices $\gamma \in \mathbb{N}^2$ such that the coefficient of $t_1^{\gamma_1} t_2^{\gamma_2}$ in $g(t_1,t_2)$ can be nonzero. Such $\gamma$ must satisfy the following inequalities
\begin{eqnarray}\label{range2}
\begin{aligned}
0 \leq a_1^1 \gamma_1 + a_1^2 \gamma_2 -2a_1^1 a_1^2 +a_1^1 +a_1^2 -2a_1^2 a_2^1 -1 \leq 2\det{(A)}-2,\\
0 \leq a_2^1 \gamma_1 + a_2^2 \gamma_2 -2a_1^2 a_2^1 +a_2^1 +a_2^2 -2 a_2^1 a_2^2 -1 \leq 2\det{(A)} -2.
\end{aligned}
\end{eqnarray}
Graphically, the region that provides a solution of both inequalities in \eqref{range2} is enclosed by a parallelogram.
\subsection{Acknowledgments}
I gratefully acknowledge my thesis advisor Professor D. Chakrabarti for the helpful comments and I thank him for the supportive suggestions. I also acknowledge the support from the dissertation committee: Prof. T. Gilsdorf, Prof. S. Narayan, and Prof. Y. Zeytuncu.

\section{Preliminary Results about Matrices}
In order to prove the main results in Section \ref{fundis}, we need the following basic lemma and a proposition that decomposes integer matrices. The proof of the following observation is trivial.
\begin{lem}\label{lemma}
Let $\Lambda \in \mathbb{M}_2(\mathbb{Z})$ be such that $\det(\Lambda)=1$. If for $p,q,r,s \in \mathbb{Z}$, we have
\begin{equation}\nonumber
\left[ 
\begin{array}{c}
p\\  
q
\end{array}
\right]= \Lambda \left[ 
\begin{array}{c}
r\\ 
s
\end{array}
\right], \quad \text{then}\; \gcd(p,q)=\gcd(r,s).
\end{equation}
\end{lem}

The next proposition is closely related to \emph{Hermite decomposition} of integer matrices. For more details, see \cite{new}.
\begin{prop}\label{prop1}
Let $A = \left[
\begin{array}{cc}
a & b\\  \nonumber
c & d
\end{array}
\right] \in \mathbb{M}_2(\mathbb{Z})$, $\det(A) >0$, $\gcd(a,c)= \gcd(b,d)=1$. Then there are unique numbers $\ell_1, \ell_2 \in \mathbb{Z}$ such that
\begin{equation}\label{LA=H}
\left[ 
\begin{array}{cc}
\ell_1 & \ell_2 \\ 
-c & a
\end{array}
\right] A = \left[ 
\begin{array}{cc}
1 & h\\ 
0 & \det(A)
\end{array}
\right], \quad \text{where} \quad h=\ell_1b+\ell_2d,
\end{equation}
and
\begin{enumerate}
\item[(1)] $0 \leq h< \det(A)$ and $\gcd(h,\det(A))=1$,
\item[(2)]  if $L=\left[ 
\begin{array}{cc}
\ell_1 & \ell_2 \\  \nonumber
-c & a
\end{array}
\right]$, then $\det(L)=1$.
\end{enumerate}
\end{prop}
\begin{proof} Since $\gcd(a,c)=1$, there exist integers $\alpha, \gamma$ such that
\begin{equation}\label{gcd}
\alpha a + \gamma c =1.
\end{equation}
Define $L_1=\left[ 
\begin{array}{cc}
\alpha & \gamma \\  \nonumber
-c & a
\end{array}
\right]$, then $L_1 A = \left[ 
\begin{array}{cc}
\alpha & \gamma \\ 
-c & a
\end{array}
\right] \left[ 
\begin{array}{cc}
a & b\\  
c & d
\end{array}
\right]=\left[ 
\begin{array}{cc}
1 & h_1\\ 
0 & \det(A)
\end{array}
\right],$ where $h_1=\alpha b +\gamma d$.

Since $\det(A)>0$, by the Division Algorithm there exist unique integers $q$ and $h$ such that
\begin{equation}\nonumber
h_1 =q\det(A) +h, \quad  0\leq h < \det(A).
\end{equation}
Let $\ell_1 = \alpha +qc$ and $\ell_2 = \gamma-aq$, and as in the statement of the proposition and recall $L= \left[ 
\begin{array}{cc}
\ell_1 & \ell_2 \\ 
-c & a
\end{array}
\right]$. Observe that $\det(L)=\ell_1 a +\ell_2 c = (\alpha +qc)a +(\gamma-aq)c= \alpha a+\gamma c =1$ by \eqref{gcd} which proves condition (2).

Observe that
\begin{eqnarray}\nonumber
\left[ 
\begin{array}{cc}
1 & h\\  \nonumber
0 & \det(A)
\end{array}
\right] &=& \left[ 
\begin{array}{cc}
1 & h_1 -q\det(A) \\  \nonumber
0 & \det(A)
\end{array}
\right] =  \left[ 
\begin{array}{cc}
1 & -q\\  \nonumber
0 & 1
\end{array}
\right] \left[ 
\begin{array}{cc}
1 & h_1 \\  \nonumber
0 & \det(A)
\end{array}
\right] =\left[ 
\begin{array}{cc}
1 & -q\\  \nonumber
0 & 1
\end{array}
\right] L_1 A \\ \nonumber
&=& \left[ 
\begin{array}{cc}
1 & -q\\  \nonumber
0 & 1
\end{array}
\right] \left[ 
\begin{array}{cc}
\alpha & \gamma \\  \nonumber
-c & a
\end{array}
\right]A =\left[ 
\begin{array}{cc}
\alpha +qc & \gamma-aq\\  \nonumber
-c & a
\end{array}
\right] A = \left[ 
\begin{array}{cc}
\ell_1 & \ell_2 \\  \nonumber
-c & a
\end{array}
\right] A =LA,
\end{eqnarray}
and \eqref{LA=H} is satisfied. So $h=\ell_1b + \ell_2d$. Since we already know that $0 \leq h < \det(A)$, the first part of condition (1) is proven.

We claim that $\gcd(h, \det(A))=1$. We can write the two equations $h=\ell_1b +\ell_2 d$ and $\det(A)=ad-cb$ in the following system
\begin{equation}\nonumber
\left[ 
\begin{array}{c}
h\\  
\det(A)
\end{array}
\right]= L \left[ 
\begin{array}{c}
b\\ 
d
\end{array}
\right].
\end{equation}
Since $\det(L)=1$ and $\gcd(b,d)=1$, then by Lemma \ref{lemma} above we have $\gcd(h,\det(A))=\gcd(b,d)=1$ which proves the second part of condition (1).

To complete the proof, we must show that $\ell_1$ and $\ell_2$ are the only numbers with properties (1) and (2). To do this, let
\begin{equation}\nonumber
\mathbb{S}=\{ (x,y)\in \mathbb{Z}^2: ax+cy=1, \: \: \: 0 \leq bx+dy < \det(A) \}.
\end{equation}
This set is nonempty because $(\ell_1,\ell_2) \in \mathbb{S}$. This implies that the general solution of the linear Diophantine equation $ax+cy=1$ is
\begin{eqnarray}\label{l11}
x= \ell_1 +nc, \quad y= \ell_2 -na,
\end{eqnarray}
where $n \in \mathbb{Z}$. If $(x,y)\in \mathbb{S}$, then $0 \leq bx+dy < \det(A)$, i.e.
\begin{eqnarray}\nonumber
0 \leq b(\ell_1 +nc) +d(\ell_2 -na) < \det(A). \nonumber
\end{eqnarray}
Simplifying the above and using $h=\ell_1b+\ell_2d$, we see that
\begin{eqnarray}\nonumber
0 \leq h-n\det(A) < \det(A),\quad \text{i.e.}\quad n\det(A) \leq h <(n+1)\det(A).
\end{eqnarray}
But we know that
\begin{equation}\label{inequ}
0 \leq h < \det(A).
\end{equation}
Notice that the sets $\{ \lambda \in \mathbb{N} \mid n\det(A) \leq \lambda <(n+1)\det(A) \}$ are disjoint for different integers $n$, so $h$ belongs to exactly one of them, which therefore corresponds to $n=0$ by \eqref{inequ}. So, $n=0$ and by \eqref{l11} we have $x=\ell_1$ and $y=\ell_2$. So $\ell_1, \ell_2$ are unique.
\end{proof}

\section{Bergman Kernel of Generalized Hartogs Triangle}
In this section, we compute the Bergman kernel of one particular $2$-dimensional monomial polyhedron, namely, the generalized Hartogs triangle $\mathcal{H}(k_1,-k_2)$ defined by
\begin{equation}\nonumber
\mathcal{H}(k_1,-k_2)= \{ (z_1,z_2) \in \mathbb{C}^2: |z_1|^{k_1} < |z_2|^{k_2} < 1 \},
\end{equation}
Edholm in \cite{luke1} (see also \cite{edho}) found the Bergman kernel of $\mathcal{H}(1,-k_2)$ and $\mathcal{H}(k_1,-1)$. In \cite{austin}, Chakrabarti, Konkel, Mainkar, and Miller give an explicit form for the Bergman kernels of generalized Hartogs triangle in $\mathbb{C}^n$, which are special cases of elementary Reinhardt domain.

The classical Hartogs triangle $\mathcal{H}(1,-1)$ is an important example in several complex variables. It is a well-known nonsmooth pseduoconvex domain.
\begin{prop}\label{prop2}
Let $k_1,k_2$ be positive integers such that $\gcd(k_1,k_2)=1$ and $0 \leq k_2 < k_1$. Then the Bergman kernel of $\mathcal{H}(k_1,-k_2)$ is given by
\begin{eqnarray}\nonumber
K_{\mathcal{H}(k_1,-k_2)}(z,w)=\dfrac{\mathlarger {\sum\limits_{\beta \in \mathbb{N}^2}}D_{k_1}(\beta_1)D_{k_1}(k_1\beta_2 + \beta_1 k_2 +k_1 +k_2 -1 -2k_1 k_2)t_1^{\beta_1} t_2^{\beta_2}}
{\pi^2 k_1 \left(t_2^{k_2}-t_1^{k_1}\right)^2  \left(1-t_2\right)^2},
\end{eqnarray}
where the sum is finite, $t_i=z_i \overline{w_i}$ for $i=1,2$, and a term the sum in the numerator can be nonzero only if  $0 \leq \beta_1 \leq 2k_1 -2$, $0 \leq \beta_2 \leq 2k_2$.
\end{prop}
\begin{proof}
Let $\mathbb{D}=\{z\in \mathbb{C}:|z| <1\}$ and $\mathbb{D}^*=\{z\in \mathbb{C}: 0<|z| <1\}$. The associated matrix $B$ of the generalized Hartogs triangle $\mathcal{H}(k_1,-k_2)$ and matrix $C=\adj( B)$ are given by
\begin{equation}\nonumber
B= \left[ 
\begin{array}{cc}
k_1 & -k_2 \\  \nonumber
0 & 1
\end{array}
\right], \quad C= \left[ 
\begin{array}{cc}
1 & k_2 \\  \nonumber
0 & k_1
\end{array}
\right].
\end{equation}
Define the map $\phi:\mathbb{D}\times \mathbb{D}^*\rightarrow \mathcal{H}(k_1,-k_2)$ by
\begin{eqnarray}
\phi(p_1,p_2) &=& (p_1p_2^{k_2},p_2^{k_1}).\label{equ2}
\end{eqnarray}
It is easy to see that $\phi$ is a proper holomorphic map of quotient type, see \cite{mass}.
To apply Bell's transformation theorem, with $\Omega_1=\mathbb{D}\times \mathbb{D}^*$ and $\Omega_2=\mathcal{H}(k_1,-k_2)$, we need to calculate the Jacobian determinant of the map $\phi$ defined in \eqref{equ2} which is given by
\begin{eqnarray}
\det \phi{'}(p_1,p_2)&=& \det \left( 
\begin{array}{cc}
p_2^{k_2} & k_2 p_1 p_2^{k_2 -1}\\ 
0 & k_1 p_2^{k_1 -1}
\end{array}
\right)= k_1 p_2^{k_1 +k_2 -1}. \label{eq3}
\end{eqnarray}

Now, we need to calculate the local branches of $\phi^{-1}$.By Theorem 3.12 in \cite{bender}, there are $\det{A}=k_1$ locally branches denoted by $\Phi_j$, $j=0, \cdots, k_1-1$. To find $\Phi_j(q)$, we need to solve the equation $\phi(p_1,p_2)=(q_1,q_2)$. We see that $p_1^{(j)} = q_1 q_2^{-k_2 /k_1} \zeta^{-jk_2}$ and $p_2^{(j)} = q_2^{1/k_1} \zeta^j$ for $j=0,\cdots,k_1-1$ where
\begin{equation}
\zeta =\exp\left(\frac{2 \pi i}{k_1}\right).
\end{equation}
It follows that we can take
\begin{eqnarray}\label{local}
\Phi_j(q_1,q_2)=(q_1 q_2^{-k_2 /k_1} \zeta^{-jk_2}, q_2^{1/k_1} \zeta^j), \quad j=0,\cdots,k_1 -1.
\end{eqnarray}
The next step is to calculate the Jacobian determinant of $\Phi_j$, which can be computed as follows
\begin{eqnarray} \nonumber
\det \Phi{'}(q_1,q_2)&=& \det \left( 
\begin{array}{cc}
q_2^{-k_2/k_1}\zeta^{-jk_2} & \frac{-k_2}{k_1}q_1 q_2^{-k_2/k_1 -1}\zeta^{-jk_2}\\  \nonumber
0 & \frac{1}{k_1}q_2^{1/k_1 -1}\zeta^{j}
\end{array}
\right)\\
& = & \frac{1}{k_1} q_2^{\frac{1-k_1 -k_2}{k_1}}\zeta^{j(1-k_2)}.\label{eq4}
\end{eqnarray}
Observe that the Bergman kernel of the unit disc and the Bergman kernel of punctured disc are the same and therefore
\begin{equation}\label{disk}
K_{\mathbb{D}\times \mathbb{D}^*}(p_1,p_2,q_1,q_2)=\frac{1}{\pi^2 (1-p_1 \overline{q_1})^2 (1-p_2 \overline{q_2})^2}.
\end{equation}
Apply Bell's Transformation law with the map $\phi$ defined in \eqref{equ2}, we get
\begin{equation}\label{bergnew}
\det{\phi'(p)}K_{\mathcal{H}(k_1,-k_2)}(\phi(p),q)=\mathlarger{\sum_{j=0}^{k_1-1}} K_{\mathbb{D}\times \mathbb{D}^*}(p,\Phi_j(q))\cdot \overline{\det{\Phi'(q)}}.
\end{equation}
Plugging \eqref{eq3}-\eqref{disk} in \eqref{bergnew}, we get
\begin{equation}\label{berg}
k_1 p_2^{k_1 +k_2 -1}K_{\mathcal{H}(k_1,-k_2)}(\phi(p),q)=\mathlarger{\sum_{j=0}^{k_1-1}}\frac{\frac{1}{k_1} \left(\overline{q_2}\right)^{\frac{1-k_1 -k_2}{k_1}}\zeta^{-j(1-k_2)}}{\pi^2(1-p_1 \overline{q_1}\: \left(\overline{q_2}\right)^{-k_2 /k_1} \zeta^{jk_2})^2(1-p_2\left(\overline{q_2}\right)^{1/k_1}\zeta^{-j})^2}.
\end{equation}
Solving for $K_{\mathcal{H}(k_1,-k_2)}(\phi(p),q)$ in \eqref{berg}, we get
\begin{eqnarray}\label{berg2}
K_{\mathcal{H}(k_1,-k_2)}(\phi(p),q)=\dfrac{\left(\overline{q_2}\right)^{\dfrac{1-k_1 -k_2}{k_1}}}{\pi^2 k_1^2 p_2^{k_1 +k_2 -1}} \sum_{j=0}^{k_1-1}\frac{\zeta^{-j(1-k_2)}}{(1-A\zeta^{jk_2})^2(1-B\zeta^{-j})^2},
\end{eqnarray}
where $A=p_1 \overline{q_1}\: \left(\overline{q_2}\right)^{-k_2 /k_1}$ and $B=p_2\left(\overline{q_2}\right)^{1/k_1}$. Our goal now is to evaluate the sum in \eqref{berg2}. To do this, observe that

\begin{equation}\nonumber
\prod_{j=0}^{k_1-1}(1-A\zeta^{jk_2})^2(1-B\zeta^{-j})^2=(1-A^{k_1})^2(1-B^{k_1})^2.
\end{equation}
So, we can rewrite \eqref{berg2} as follows
\begin{eqnarray}\label{berg3}
K_{\mathcal{H}(k_1,-k_2)}(\phi(p),q)=\dfrac{\left(\overline{q_2}\right)^{\frac{1-k_1 -k_2}{k_1}}}{\pi^2 k_1^2 p_2^{k_1 +k_2 -1}} \frac{h(A,B)}{(1-A^{k_1})^2(1-B^{k_1})^2},
\end{eqnarray}
where
\begin{eqnarray}\label{hfun}
h(A,B)=\sum_{j=0}^{k_1-1}\left(\dfrac{1-A^{k_1}}{1-A\zeta^{jk_2}}\right)^2 \left(\dfrac{1-B^{k_1}}{1-B \zeta^{-j}}\right)^2 \zeta^{-j(1-k_2)}.
\end{eqnarray}
We claim that $h$ satisfies the following symmetry property
\begin{equation}\label{sym}
h(\zeta^{k_2}A,\zeta^{-1}B)=\zeta^{1-k_2}h(A,B).
\end{equation}
To see this, note that
\begin{eqnarray}\nonumber
h(\zeta^{k_2}A,\zeta^{-1}B) &=& \mathlarger{\sum_{j=0}^{k_1-1}}\left(\dfrac{1-(\zeta^{k_2}A)^{k_1}}{1-(\zeta^{k_2}A)\zeta^{jk_2}}\right)^2 \left(\dfrac{1-(\zeta^{-1}B)^{k_1}}{1-(\zeta^{-1}B) \zeta^{-j}}\right)^2 \zeta^{-j(1-k_2)}\\ \nonumber
&=& \mathlarger{\sum_{j=0}^{k_1-1}}\left(\dfrac{1-A^{k_1}}{1-\zeta^{k_2(1+j)}A}\right)^2 \left(\frac{1-B^{k_1}}{1-\zeta^{-(1+j)}B)}\right)^2 \zeta^{-j(1-k_2)} \\ \label{symmetry}
&=& \mathlarger{ \sum_{\ell=1}^{k_1}}\left(\dfrac{1-A^{k_1}}{1-\zeta^{k_2\ell}A}\right)^2 \left(\dfrac{1-B^{k_1}}{1-\zeta^{-\ell}B)}\right)^2 \zeta^{-(\ell-1)(1-k_2)} \\ \nonumber
&=& \zeta^{(1-k_2)}h(A,B).
\end{eqnarray}
Equation \eqref{symmetry} holds by letting $\ell=1+j$ and the last equality holds by rearranging the terms and noting that the general term is the same when $\ell=k_1$ and when $\ell=0$, and this completes the proof of the symmetry of $h$ in \eqref{sym}.\\
Now, note that $\dfrac{1-A^{k_1}}{1-\zeta^{k_2\ell}A}$ is a polynomial of degree $k_1-1$ in $A$, and so $\left(\dfrac{1-A^{k_1}}{1-\zeta^{k_2 \ell}A}\right)^2$ is a polynomial of degree $2k_1-2$ in $A$. It follows that
\begin{equation}\nonumber
h(A,B)=\sum_{\alpha_1=0}^{2k_1-2}\sum_{\alpha_2=0}^{2k_1-2}h_{\alpha_1,\alpha_2}A^{\alpha_1}B^{\alpha_2}.
\end{equation}
Here, $h_{\alpha_1,\alpha_2}$ depends on $\zeta,k_2$ and $j$. By the symmetry of $h$ as in \eqref{sym}, we have
\begin{equation}
\sum_{\alpha_1=0}^{2k_1-2}\sum_{\alpha_2=0}^{2k_1-2}h_{\alpha_1,\alpha_2}(\zeta^{k_2}A)^{\alpha_1}(\zeta^{-1}B)^{\alpha_2} =\zeta^{1-k_2} \sum_{\alpha_1=0}^{2k_1-2}\sum_{\alpha_2=0}^{2k_1-2}h_{\alpha_1,\alpha_2}A^{\alpha_1}B^{\alpha_2}, \nonumber
\end{equation}
it follows that
\begin{equation}\nonumber
\sum_{\alpha_1=0}^{2k_1-2}\sum_{\alpha_2=0}^{2k_1-2}h_{\alpha_1,\alpha_2}(\zeta^{k_2\alpha_1-\alpha_2}-\zeta^{1-k_2})A^{\alpha_1}B^{\alpha_2}=0,
\end{equation}
so, if $h_{\alpha_1,\alpha_2}\ne 0$, we should have
\begin{equation}\nonumber
\zeta^{k_2\alpha_1-\alpha_2+k_2-1}=1.
\end{equation}
Thus, $k_1\vert (k_2\alpha_1-\alpha_2+k_2-1)$, since $\zeta=e^{2\pi i /{k_1}}$ i.e. there is $\nu \in \mathbb{Z}$ such that
\begin{equation}\label{cond}
k_2\alpha_1-\alpha_2+k_2-1=\nu k_1.
\end{equation}
Substituting $x$ by $\zeta^{jk_2}A$ in \eqref{event} in Lemma \ref{lemma1} we get
\begin{equation}\label{eq5}
\left(\dfrac{1-A^{k_1}}{1-\zeta^{jk_2}A} \right)^2 = \mathlarger{\sum_{\alpha_1=0}^{2k_1-2}}D_{k_1}(\alpha_1)A^{\alpha_1}\zeta^{jk_2\alpha_1},
\end{equation}
Similarly, we can write
\begin{equation}\label{eq6}
\left(\dfrac{1-B^{k_1}}{1-\zeta^{-j}B} \right)^2 = \mathlarger{\sum_{\alpha_2=0}^{2k_1-2}}D_{k_1}(\alpha_2)B^{\alpha_2}\zeta^{-j\alpha_2}.
\end{equation}
Plugging \eqref{eq5} and \eqref{eq6} in the formula \eqref{hfun} for $h(A,B)$, we get
\begin{eqnarray}\nonumber
h(A,B)&=&\mathlarger{\sum\limits_{j=0}^{k_1-1}}\left(\mathlarger{ \sum\limits_{\alpha_1=0}^{2k_1-2}}D_{k_1}(\alpha_1)A^{\alpha_1}\zeta^{jk_2\alpha_1} \right) \left(\mathlarger{\sum\limits_{\alpha_2=0}^{2k_1-2}}D_{k_1}(\alpha_2)B^{\alpha_2}\zeta^{-j\alpha_2} \right)\zeta^{-j(1-k_2)}\\ \nonumber
&=& \mathlarger{\sum\limits_{j=0}^{k_1-1}} \: \mathlarger{\sum\limits_{\alpha_1=0}^{2k_1-2}}\: \mathlarger{\sum\limits_{\alpha_2=0}^{2k_1-2}}D_{k_1}(\alpha_1)D_{k_1}(\alpha_2)A^{\alpha_1}B^{\alpha_2}\zeta^{j(k_2\alpha_1-\alpha_2-1+k_2)}\\ \nonumber
&=& \mathlarger{\sum\limits_{j=0}^{k_1-1}} \mathlarger{\sum\limits_{\nu \in \mathbb{Z}}}\: \mathlarger{ \sum\limits_{\alpha_1,\alpha_2=0}^{2k_1-2}}D_{k_1}(\alpha_1)D_{k_1}(\alpha_2)A^{\alpha_1}B^{\alpha_2}\zeta^{j k_1 \nu}, \quad \text{since} \: k_2\alpha_1-\alpha_2-1+k_2 =\nu k_1, \nu \in \mathbb{Z} \\ \label{hfun25}
&=& \mathlarger{\sum\limits_{j=0}^{k_1-1}} \mathlarger{ \sum\limits_{\substack{\alpha_1,\alpha_2=0\\ \nu \in \mathbb{Z}}}^{2k_1-2}} \mathlarger{\sum\limits_{\substack{k_2(\alpha_1+1)\\-(\alpha_2+1)=\nu k_1}}}D_{k_1}(\alpha_1)D_{k_1}(\alpha_2)A^{\alpha_1}B^{\alpha_2}, \quad \quad \quad \text{since}\: \zeta=e^{2\pi i/k_1},\: \text{so} \: \zeta^{jk_1 \nu}=1 \\ 
&=& k_1 \mathlarger{\sum\limits_{\substack{\alpha_1,\alpha_2=0\\ \nu \in \mathbb{Z}}}^{2k_1-2}}\mathlarger{\sum\limits_{\substack{k_2(\alpha_1+1)\\-(\alpha_2+1)=\nu k_1}}}D_{k_1}(\alpha_1)D_{k_1}(\alpha_2)A^{\alpha_1}B^{\alpha_2}. \label{hfun2}
\end{eqnarray}
We get $k_1$ in \eqref{hfun2} since the summand in \eqref{hfun25} is free of $j$. Then we plug \eqref{hfun2} in \eqref{berg3} and get
\begin{eqnarray}\label{berg4}
K_{\mathcal{H}(k_1,-k_2)}(\phi(p),q)=\dfrac{\left(\overline{q_2}\right)^{\frac{1-k_1 -k_2}{k_1}}}{\pi^2 k_1 p_2^{k_1 +k_2 -1}}\cdot
\dfrac{\mathlarger{\sum\limits_{\substack{\alpha_1,\alpha_2=0 \\ \nu \in \mathbb{Z}}}^{2k_1-2}}\mathlarger{\sum\limits_{\substack{k_2(\alpha_1+1)\\-(\alpha_2+1)=\nu k_1}}}D_{k_1}(\alpha_1)D_{k_1}(\alpha_2)A^{\alpha_1}B^{\alpha_2}}{(1-A^{k_1})^2(1-B^{k_1})^2}.
\end{eqnarray}
Using the values of $A=p_1 \overline{q_1}\: \left(\overline{q_2}\right)^{-k_2 /k_1}$ and $B=p_2\left(\overline{q_2}\right)^{1/k_1}$ in \eqref{berg4}, we get

\begin{eqnarray}\label{berg5}\nonumber
K_{\mathcal{H}(k_1,-k_2)}(\phi(p),q)&=& \dfrac{\left(\overline{q_2}\right)^{\frac{1-k_1 -k_2}{k_1}}\mathlarger{\sum\limits_{\substack{\alpha_1,\alpha_2=0\\ \nu \in \mathbb{Z}}}^{2k_1-2}}\mathlarger{\sum\limits_{\substack{k_2(\alpha_1+1)\\-(\alpha_2+1)=\nu k_1}}}D_{k_1}(\alpha_1)D_{k_1}(\alpha_2)\left( p_1 \overline{q_1}\left(\overline{q_2}\right)^{-k_2 /k_1}\right)^{\alpha_1} \left(p_2\left(\overline{q_2}\right)^{1/{k_1}}\right)^{\alpha_2}}{\pi^2 k_1 p_2^{k_1 +k_2 -1}\left(1-\left( p_1 \overline{q_1}\left(\overline{q_2}\right)^{-k_2 /k_1}\right)^{k_1}\right)^2\left(1-\left( p_2 \left(\overline{q_2}\right)^{1/k_1}\right)^{k_1}\right)^2}\\ \nonumber
&=&\dfrac{\mathlarger{\sum\limits_{\substack{\alpha_1,\alpha_2=0\\ \nu \in \mathbb{Z}}}^{2k_1-2}}\mathlarger{\sum\limits_{\substack{k_2(\alpha_1+1)\\-(\alpha_2+1)=\nu k_1}}}D_{k_1}(\alpha_1)D_{k_1}(\alpha_2)\left( p_1^{k_1} \left(\overline{q_1}\right)^{k_1}\left(\overline{q_2}\right)^{-k_2}\right)^{\alpha_1 / k_1} \left(p_2^{k_1}\overline{q_2}\right)^{\alpha_2/{k_1}}
 }{\pi^2 k_1 \left(p_2^{k_1}\overline{q_2}\right)^{(k_1+k_2-1)/{k_1}}\left(1-p_1^{k_1}\left(\overline{q_1}\right)^{k_1}\left(\overline{q_2}\right)^{-k_2}\right)^2\left(1-p_2^{k_1}\overline{q_2}\right)^2}\\ \label{berg55}
&=& K_{\mathcal{H}(k_1,-k_2)}(p_1p_2^{k_2},p_2^{k_1},q_1,q_2).
\end{eqnarray}
Let $z_1=p_1 p_2^{k_2}$, $z_2=p_2^{k_1}$, $w_1=q_1$ and $w_2=q_2$ in \eqref{berg55}, we get
\begin{eqnarray}\label{berg66}
K_{\mathcal{H}(k_1,-k_2)}(z,w)=\dfrac{\mathlarger{\sum\limits_{\substack{\alpha_1,\alpha_2=0\\ \nu \in \mathbb{Z}}}^{2k_1-2}}\mathlarger{\sum\limits_{\substack{k_2(\alpha_1+1)\\-(\alpha_2+1)=\nu k_1}}} D_{k_1}(\alpha_1)D_{k_1}(\alpha_2)\left(\frac{z_1^{k_1}\overline{w_1}^{k_1}}{(z_2\overline{w_2})^{k_2}} \right)^{\alpha_1 / k_1} \left(z_2\overline{w_2}\right)^{\alpha_2/{k_1}}
 }{\pi^2 k_1 \left(z_2\overline{w_2}\right)^{(k_1+k_2-1)/{k_1}}\left(\frac{(z_2\overline{w_2})^{k_2}-z_1^{k_1}(\overline{w_1})^{k_1}}{(z_2\overline{w_2})^{k_2}}  \right)^2\left(1-z_2\overline{w_2}\right)^2}.
\end{eqnarray}
Let $t_1=z_1\overline{w_1}$ and $t_2=z_2\overline{w_2}$ in \eqref{berg66}, this implies
\begin{eqnarray}\nonumber
K_{\mathcal{H}(k_1,-k_2)}(z,w)&=&\dfrac{\mathlarger{\sum\limits_{\substack{\alpha_1,\alpha_2=0\\ \nu \in \mathbb{Z}}}^{2k_1-2}}\mathlarger{\sum\limits_{\substack{k_2(\alpha_1+1)\\-(\alpha_2+1)=\nu k_1}}} D_{k_1}(\alpha_1)D_{k_1}(\alpha_2)\dfrac{(t_1^{k_1})^{\alpha_1 / k_1}}{(t_2^{k_2})^{\alpha_1 / k_1}}  t_2^{\alpha_2/{k_1}}
 }{\pi^2 k_1 t_2^{(k_1+k_2-1)/{k_1}}\frac{\left(t_2^{k_2}-t_1^{k_1}\right)^2}{t_2^{2k_2}}  \left(1-t_2\right)^2}\\ 
&=&\dfrac{\mathlarger{\sum\limits_{\substack{\alpha_1,\alpha_2=0\\ \nu \in \mathbb{Z}}}^{2k_1-2}}\mathlarger{\sum\limits_{\substack{k_2(\alpha_1+1)\\-(\alpha_2+1)=\nu k_1}}}D_{k_1}(\alpha_1)D_{k_1}(\alpha_2)t_1^{\alpha_1} t_2^{\frac{\alpha_2 -\alpha_1 k_2 -k_1-k_2 +1 +2k_1 k_2}{k_1}}}
{\pi^2 k_1 \left(t_2^{k_2}-t_1^{k_1}\right)^2  \left(1-t_2\right)^2}. \label{shit0}
\end{eqnarray}
The numerator in \eqref{shit0} can be written as
\begin{eqnarray}\nonumber
&=&\mathlarger{\sum\limits_{\substack{\alpha_1,\alpha_2=0\\ \nu \in \mathbb{Z}}}^{2k_1-2}}\mathlarger{\sum\limits_{\substack{k_2(\alpha_1+1)\\-(\alpha_2+1)=\nu k_1}}} D_{k_1}(\alpha_1)D_{k_1}(\alpha_2)t_1^{\alpha_1} t_2^{\frac{(\alpha_2+1) -k_2(\alpha_1 +1) -k_1(1-2k_2)}{k_1}}\\ \nonumber
&=&\mathlarger{\sum\limits_{\substack{\alpha_1,\alpha_2=0\\ \nu \in \mathbb{Z}}}^{2k_1-2}}\mathlarger{\sum\limits_{\substack{k_2(\alpha_1+1)\\-(\alpha_2+1)=\nu k_1}}} D_{k_1}(\alpha_1)D_{k_1}(\alpha_2)t_1^{\alpha_1} t_2^{\frac{-k_1 \nu -k_1(1-2k_2)}{k_1}} 
\end{eqnarray}

\begin{eqnarray}
=\mathlarger{\sum\limits_{\substack{\alpha_1,\alpha_2=0\\ \nu \in \mathbb{Z}}}^{2k_1-2}}\mathlarger{\sum\limits_{\substack{k_2(\alpha_1+1)\\-(\alpha_2+1)=\nu k_1}}} D_{k_1}(\alpha_1)D_{k_1}(\alpha_2)t_1^{\alpha_1} t_2^{-\nu -(1-2k_2)}.\label{shit}
\end{eqnarray}

Plugging \eqref{shit} in \eqref{shit0} we get
\begin{eqnarray}
K_{\mathcal{H}(k_1,-k_2)}(z,w)&=&\dfrac{\mathlarger{\sum\limits_{\substack{\alpha_1,\alpha_2=0\\ \nu \in \mathbb{Z}}}^{2k_1-2}}\mathlarger{\sum\limits_{\substack{k_2(\alpha_1+1)\\-(\alpha_2+1)=\nu k_1}}} D_{k_1}(\alpha_1)D_{k_1}(\alpha_2)t_1^{\alpha_1} t_2^{-\nu -(1-2k_2)}}{\pi^2 k_1 \left(t_2^{k_2}-t_1^{k_1}\right)^2  \left(1-t_2\right)^2}. \label{berg99}
\end{eqnarray}

Let $\beta_2 = \dfrac{1}{k_1}(\alpha_2 -\alpha_1 k_2 -k_1-k_2 +1 +2k_1 k_2)$. We know from \eqref{cond} that
\begin{equation}\nonumber
k_2\alpha_1-\alpha_2+k_2-1=\nu k_1.
\end{equation}

So $\beta_2=-\nu -(1-2k_2)$ is an integer. Let $\beta_1 = \alpha_1$, so from \eqref{berg99} we have
\begin{eqnarray}\label{berg100}
K_{\mathcal{H}(k_1,-k_2)}(z,w)&=&\dfrac{1}{\pi^2 k_1}\cdot \dfrac{\mathlarger{\sum\limits_{\beta \in \mathbb{N}^2}}D_{k_1}(\beta_1)D_{k_1}(k_1\beta_2 + \beta_1 k_2 +k_1 +k_2 -1 -2k_1 k_2)t_1^{\beta_1} t_2^{\beta_2}}
{ \left(t_2^{k_2}-t_1^{k_1}\right)^2  \left(1-t_2\right)^2}.
\end{eqnarray}
Since $0 \leq \alpha_1, \alpha_2 \leq 2k_1 -2$, it follows from the definition of $\beta_1$ and $\beta_2$ that
\begin{eqnarray}\label{nbeta1}
0 & \leq & \beta_1 \leq 2k_1 -2, \\ \label{nbeta2}
0 & \leq & k_2 \beta_1 + k_1 \beta_2 + k_1 + k_2  - 1 - 2k_1 k_2 \leq 2k_1 -2.
\end{eqnarray}
So the sum in \eqref{berg100} is finite. This completes the proof of Proposition \eqref{prop2}.
\end{proof}
\section{Proof of Theorem \ref{mainthm1} and Theorem \ref{mainthm2}}
Now, we are ready to prove Theorem \ref{mainthm1} and Theorem \ref{mainthm2} and calculate the Bergman Kernel of the domain in \eqref{defdef}
\begin{equation}\nonumber
\mathcal{U}= \left\{ (z_1,z_2) \in \mathbb{C}^2: |z_1|^{b_1^1} |z_2|^{b_2^1}<1,\: |z_1|^{b_1^2} |z_2|^{b_2^2}<1 \right\}.
\end{equation} 
We will use the following result.
\begin{thm}[\cite{bender},Theorem 3.12]\label{benthm}
Suppose that the monomial polyhedron $\mathcal{U}$, the domain defined in Definition \ref{dom1}, is bounded, and assume that $B$ is the associated matrix of $\mathcal{U}$, and let $C=\adj{(B)}$. Then the map $\phi_C(z)=z^C$ maps $\mathbb{D}_{L(B)}^n$ onto $\mathcal{U}$, and the map so defined
\begin{equation}\nonumber
\phi_C: \mathbb{D}_{L(B)}^n \rightarrow \mathcal{U}
\end{equation}
is a proper holomorphic map of quotient type with group $\Gamma$ consisting of the automorphisms
\begin{equation}\nonumber
\sigma_v: \mathbb{D}_{L(B)}^n \rightarrow \mathbb{D}_{L(B)}^n
\end{equation}
given by
\begin{equation}\nonumber
\sigma_v(z)= \exp \left(2 \pi i C^{-1} \right) \odot v
\end{equation}
for $v \in \mathbb{Z}^n$. Further, the group $\Gamma$ has exactly $\det(C)$ elements.
\end{thm}
Recall that
\begin{eqnarray} \nonumber
B &=& \left[ 
\begin{array}{cc}
b_1^1 & b_2^1\\  \nonumber
b_1^2 & b_2^2
\end{array}
\right] \label{matB}
\end{eqnarray}
is the associated matrix for the domain $\mathcal{U}$.
It is a well-known fact in Algebra that $B^{-1}=\frac{1}{\det{B}}\cdot \adj(B)$, where $\adj(B)=\left[ \begin{array}{cc}
b_2^2 & -b_2^1\\  \nonumber
-b_1^2 & b_1^1
\end{array}
\right]=C$ represents the \emph{adjugate matrix} of $B$. The greatest common factor of the entries for each column in $C$ is not necessarily equal to $1$. So, we are going to divide each entry in $C$ by the greatest common divisor of its column's entries. This introduces the matrix $A=\left[ \begin{array}{cc}
a_1^1 & a_2^1\\  \nonumber
a_1^2 & a_2^2
\end{array}
\right]$ as in \eqref{matAAA}. In particular,
\begin{eqnarray}\nonumber
A &=& \left[ \begin{array}{cc}
a_1^1 & a_2^1\\  
a_1^2 & a_2^2
\end{array}
\right] = \left[ \begin{array}{cc}
\dfrac{b_2^2}{\gcd(b_2^2,b_1^2)} & \dfrac{-b_2^1}{\gcd(b_2^1,b_1^1)}\\ 
\dfrac{-b_1^2}{\gcd(b_2^2,b_1^2)} & \dfrac{b_1^1}{\gcd(b_2^1,b_1^1)}
\end{array}
\right].
\end{eqnarray}
We can write $C$ as 
\begin{eqnarray}\nonumber
C = AG, \quad \text{where}\quad G=\left[ \begin{array}{cc}
\gcd(b_2^2,b_1^2) & 0\\  
0 & \gcd(b_2^1,b_1^1)
\end{array}
\right].
\end{eqnarray}
Observe that
\begin{equation}\nonumber
\det(A) = \dfrac{b_2^2 b_1^1}{\gcd(b_2^2, b_1^2) \gcd(b_2^1, b_1^1)}-\dfrac{b_2^1 b_1^2}{\gcd(b_2^2, b_1^2) \gcd(b_2^1, b_1^1)}=\dfrac{\det(B)}{\gcd(b_2^2, b_1^2) \gcd(b_2^1, b_1^1)}.
\end{equation}
So, $\det(A)>0$ since $\det(B)>0$ by hypothesis.

Applying Proposition \ref{prop1} to the matrix $A$, and letting $\ell_1$, $\ell_2$ and $h$ be as in Proposition \ref{prop1}, we have $A=L^{-1}H$, where

\begin{equation}
L=\left[ \begin{array}{cc}
\ell_1 & \ell_2\\  \nonumber
-a_1^2 & a_1^1
\end{array}
\right] \quad \text{and} \quad H=\left[ \begin{array}{cc}
1 & h\\  \nonumber
0 & \det(A)
\end{array}
\right],
\end{equation}
with $h=\ell_1 a_2^1+\ell_2a_2^2$ and $\det(L)=\ell_1 a_1^1 + \ell_2 a_1^2 =1$, and $0 \leq h<\det (A)$.
Notice that
\begin{eqnarray} \label{decomp} 
\displaystyle \phi_A(z) &=& z^A = z^{L^{-1}H} = \left(z^H\right)^{L^{-1}} = \phi_{L^{-1}}(z^H) = \phi_{L^{-1}} (\phi_H(z)) = \left( \phi_{L^{-1}} \circ \phi_H \right) (z). 
\end{eqnarray}
The map $\phi_C:\mathbb{D}\times \mathbb{D}^* \rightarrow \mathcal{U}$ is a proper holomorphic map of quotient type, see Theorem \ref{benthm}, and $C=AG$, so
\begin{eqnarray}\nonumber
\phi_C = \phi_{AG} = \phi_A \circ \phi_G, \quad \text{by}\; \eqref{decomp},
\end{eqnarray}
but $\phi_G(z_1,z_2)= (z_1^{\gcd(b_2^2,b_1^2)},z_2^{\gcd(b_2^1,b_1^1)})\in \mathbb{D}\times \mathbb{D}^*$, it follows that $\phi_G(\mathbb{D}\times \mathbb{D}^*) = \mathbb{D}\times \mathbb{D}^*$. Further, $\phi_G$ is a proper holomorphic map of quotient type. We claim that $\phi_A: \mathbb{D}\times \mathbb{D}^*  \rightarrow \mathcal{U}$ is a proper holomorphic map. It suffices to show that for each compact $K \subset \mathcal{U}$, $\phi_A^{-1}(K) \subset \mathbb{D}\times \mathbb{D}^*$ is compact. Suppose there was a set $K$ for which $\phi_A^{-1}(K) \subset \mathbb{D}\times \mathbb{D}^*$ is noncompact. Then from the representation $\phi_G(z_1,z_2)=(z_1^{\gcd(b_2^2,b_1^2)},z_2^{\gcd(b_2^1,b_1^1)})$, it follows that $\phi_G^{-1}(\phi_A^{-1}(K)) \subset \mathbb{D}\times \mathbb{D}^*$ is also noncompact, that is $\phi_C^{-1}(K)$ is not compact which contradicts the properness of $\phi_C$.
By \eqref{decomp} the proper holomorphic map $\phi_A:\mathbb{D}\times \mathbb{D}^{*} \rightarrow \mathcal{U}$ factors into the proper holomorphic map $\phi_H$ and the biholomorphism $\phi_{L^{-1}}$. The diagram below shows the three maps.

\begin{equation}\nonumber
\mathbb{D}\times \mathbb{D}^{*} \xrightarrow[\text{proper holo.}]{\quad \quad \quad \phi_H \quad \quad \quad } \phi_H(\mathbb{D}\times \mathbb{D}^{*}) \xrightarrow[\text{biholo.}]{\quad \quad \quad \phi_{L^{-1}}\quad} \mathcal{U}
\end{equation}
where $\phi_H(\mathbb{D}\times \mathbb{D}^{*})$ is a generalized Hartogs triangle with $H=A$ using Proposition \ref{prop2}.

The Bergman kernel of $\phi_H(\mathbb{D}\times \mathbb{D}^{*})$, since $\gcd(\det(A),h)=1$ and $0\leq h < \det(A)$ we apply Proposition \ref{prop2} to $\phi_H$ with
\begin{equation}\nonumber
k_1=\det(A) \quad \text{and} \quad k_2=h.
\end{equation}
We get
\begin{eqnarray}\label{bergV}
K_{\phi(\mathbb{D}\times \mathbb{D}^{*})}(z,w)=\dfrac{\mathlarger{\sum\limits_{\beta \in \mathbb{N}^2}}D_{k_1}(\beta_1)D_{k_1}(k_2 \beta_1 + k_1 \beta_2 +k_1+k_2-2k_1 k_2 -1)u_1^{\beta_1} u_2^{\beta_2}}
{\pi^2 k_1 \left(u_2^{k_2}-u_1^{k_1}\right)^2  \left(1-u_2\right)^2},
\end{eqnarray}
where $u_j=z_j \overline{w_j}$, $j=1,2$ and the sum is finite and $\beta_1$ and $\beta_2$ satisfy \eqref{nbeta1} and \eqref{nbeta2}, that is
\begin{eqnarray}\label{nbeta111}
0 & \leq & \beta_1 \leq 2k_1 -2, \\ \label{nbeta222}
0 & \leq & k_2 \beta_1 + k_1 \beta_2 + k_1 + k_2  - 1 - 2k_1 k_2 \leq 2k_1 -2.
\end{eqnarray}

To calculate the Bergman kernel of $\mathcal{U}$, we apply transformation formula on the biholomorphism $\phi_L$ to get
\begin{equation}\label{bergUa}
K_{\mathcal{U}}(p,q)=\det \phi_L'(p) K_V(\phi_L(p),\phi_L(q)) \overline{\det \phi_L'(q)}.
\end{equation}

Observe that $\phi_L(p)=(p_1^{\ell_1}p_2^{\ell_2},p_1^{-a_1^2}p_2^{a_1^1})$, $\phi_L(q)=(q_1^{\ell_1}q_2^{\ell_2},q_1^{-a_1^2}q_2^{a_1^1})$, so
\begin{eqnarray}\label{moh}
\det \phi_L'(p)= p_1^{\ell_1-a_1^2-1}p_2^{\ell_2 +a_1^1 -1}, \quad \det \phi_L'(q)= q_1^{\ell_1-a_1^2-1}q_2^{\ell_2 +a_1^1 -1}.
\end{eqnarray}
Plugging \eqref{moh} in \eqref{bergUa} we get

\begin{equation}\label{bergU}
K_{\mathcal{U}}(p,q)= (p_1\overline{q_1})^{\ell_1-a_1^2-1} (p_2\overline{q_2})^{\ell_2+a_1^1-1} K_V(\phi_L(p),\phi_L(q)).
\end{equation}

Letting $z=(p_1^{\ell_1}p_2^{\ell_2},p_1^{-a_1^2}p_2^{a_1^1})$ and $w=(q_1^{\ell_1}q_2^{\ell_2},q_1^{-a_1^2}q_2^{a_1^1})$ in \eqref{bergV} we get $K_V(p_1^{\ell_1}p_2^{\ell_2},p_1^{-a_1^2}p_2^{a_1^1},q_1^{\ell_1}q_2^{\ell_2},q_1^{-a_1^2}q_2^{a_1^1})=$
\begin{eqnarray}\label{bergV1}
\dfrac{\mathlarger{\sum\limits_{\beta \in \mathbb{N}^2}}D_{k_1}(\beta_1)D_{k_1}(k_2 \beta_1 + k_1 \beta_2 +k_1+k_2-2k_1 k_2 -1)s_1^{\beta_1} s_2^{\beta_2}}
{\pi^2 k_1 \left(s_2^{k_2}-s_1^{k_1}\right)^2  \left(1-s_2\right)^2},
\end{eqnarray}
where
\begin{eqnarray}\nonumber
s_1 = (p_1\overline{q_1})^{\ell_1}(p_2\overline{q_2})^{\ell_2}, \quad s_2 = (p_1\overline{q_1})^{-a_1^2}(p_2\overline{q_2})^{a_1^1},
\end{eqnarray}
and $0 \leq \beta_1 \leq 2k_1 -2$ and $0 \leq k_2 \beta_1 + k_1 \beta_2 +k_1+k_2-2k_1 k_2 -1 \leq 2k_1-2$.
Plugging \eqref{bergV1} in \eqref{bergU} we get $K_{\mathcal{U}}(p,q)=$
\begin{equation}\label{bergU1} 
 \dfrac{\mathlarger{\sum\limits_{\beta \in \mathbb{N}^2}}D_{k_1}(\beta_1)D_{k_1}(k_2 \beta_1 + k_1 \beta_2 +k_1+k_2-2k_1 k_2 -1)t_1^{\ell_1\beta_1-a_1^2\beta_2+\ell_1-a_1^2-1} t_2^{\ell_2\beta_1+a_1^1\beta_2+\ell_2+a_1^1-1}}{\pi^2 k_1 \left(t_1^{-k_2 a_1^2}t_2^{k_2 a_1^1}-t_1^{\ell_1 k_1}t_2^{\ell_2 k_1}\right)^2 \left(1-t_1^{-a_1^2}t_2^{a_1^1} \right)^2},
\end{equation}
where $t_j=p_j\overline{q_j}$, $j=1,2$, and the sum in \eqref{bergU1} is finite and $\beta_1$ and $\beta_2$ satisfy \eqref{nbeta111} and \eqref{nbeta222}.
In order to simplify the denominator, since $\ell_1 a_1^1 +\ell_2 a_1^2=1$, we have
\begin{eqnarray}\nonumber
\ell_1 k_1+k_2 a_1^2 &=& \ell_1\det A + h a_1^2 = \ell_1(a_1^1 a_2^2 -a_1^2 a_2^1)+ (\ell_1 a_2^1 +\ell_2 a_2^2)a_1^2\\
&=& \ell_1 a_1^1 a_2^2 + \ell_2 a_2^2 a_1^2 = a_2^2 (\ell_1 a_1^1 + \ell_2 a_1^2) = a_2^2, \label{cond1}
\end{eqnarray}
and

\begin{eqnarray}\nonumber
\ell_2 k_1 -k_2 a_1^1 &=& \ell_2 \det A-ha_1^1 = \ell_2(a_1^1 a_2^2 -a_1^2 a_2^1) - (\ell_1 a_2^1 + \ell_2 a_2^2)a_1^1 \\
&=&  -\ell_2 a_1^2 a_2^1 -\ell_1 a_2^1 a_1^1 = -a_2^1 (\ell_2 a_1^2 + \ell_1 a_1^1) = -a_2^1. \label{cond2}
\end{eqnarray}
Working on the denominator in \eqref{bergU1}, and using \eqref{cond1} and \eqref{cond2} we have
\begin{eqnarray}\nonumber
\left(t_1^{-k_2 a_1^2}t_2^{k_2 a_1^1}-t_1^{\ell_1 k_1}t_2^{\ell_2 k_1}\right)^2 \left(1-t_1^{-a_1^2}t_2^{a_1^1} \right)^2 &=& t_1^{-2k_2 a_1^2 - 2a_1^2}t_2^{2k_2 a_1^1}  \left(1-t_1^{\ell_1 k_1 +k_2 a_1^2}t_2^{\ell_2 k_1 -k_2 a_1^1}\right)^2 \left(t_1^{a_1^2}-t_2^{a_1^1} \right)^2 \\ \nonumber
&=& t_1^{-2k_2 a_1^2 - 2a_1^2}t_2^{2k_2 a_1^1}  \left(1-t_1^{a_2^2}t_2^{-a_2^1}\right)^2 \left(t_1^{a_1^2}-t_2^{a_1^1} \right)^2 \\ \label{bergU2}
&=& t_1^{-2k_2 a_1^2 - 2a_1^2}t_2^{2k_2 a_1^1 -2a_2^1}  \left(t_2^{a_2^1}-t_1^{a_2^2}\right)^2 \left(t_1^{a_1^2}-t_2^{a_1^1} \right)^2.
\end{eqnarray}

Plugging \eqref{bergU2} in \eqref{bergU1}, we get

\begin{equation}\label{bergU33}
K_{\mathcal{U}}(p,q)= \dfrac{\mathlarger{\sum\limits_{\gamma \in \mathbb{Z}^2}}D_{k_1}(\beta_1)D_{k_1}(\beta_2 k_1 + \beta_1 k_2 +k_1+k_2-2k_1 k_2 -1)t_1^{\gamma_1} t_2^{\gamma_2}}{\pi^2 k_1 \left(t_2^{a_2^1}-t_1^{a_2^2}\right)^2 \left(t_1^{a_1^2}-t_2^{a_1^1} \right)^2},
\end{equation}
where
\begin{eqnarray}\nonumber
\gamma_1 &=& \ell_1\beta_1- a_1^2 \beta_2 +\ell_1 + 2k_2 a_1^2 + a_1^2 -1, \\ \nonumber
\gamma_2 &=& \ell_2 \beta_1 + a_1^1 \beta_2 + \ell_2 - 2 k_2 a_1^1 + a_1^1 + 2a_2^1 -1.
\end{eqnarray}
and the sum in \eqref{bergU33} is finite and $\beta_1$ and $\beta_2$ satisfy \eqref{nbeta1} and \eqref{nbeta2}.

Solving the expressions $\gamma_1$ and $\gamma_2$ for $\beta_1$ and $\beta_2$, we get
\begin{eqnarray}\label{beta11}
\beta_1 &=& a_1^1 \gamma_1 + a_1^2 \gamma_2 -2a_1^1 a_1^2 +a_1^1 +a_1^2 -2a_1^2 a_2^1 -1, \\ \label{beta22}
\beta_2 &=& -\ell_2 \gamma_1 + \ell_1 \gamma_2 + \ell_2 a_1^2 - \ell_1 a_1^1 +2 k_2 - \ell_2 + \ell_1 - 2 a_2^1 \ell_1.
\end{eqnarray}

Define $\eta =\beta_1 k_2 + \beta_2 k_1 + k_1 + k_2 -2 k_1 k_2 -1$. Plugging the values of $\beta_1$ and $\beta_2$ above in  $\eta$ and using \eqref{cond1} and \eqref{cond2} we get
\begin{dmath} \label{eta}
\eta = \beta_1 k_2 + \beta_2 k_1 + k_1 + k_2 -2 k_1 k_2 -1  \\ \nonumber
     = \left(a_1^1 \gamma_1 + a_1^2 \gamma_2 -2a_1^1 a_1^2 +a_1^1 +a_1^2 -2a_1^2 a_2^1 -1 \right)k_2 +     
\left(-\ell_2 \gamma_1 + \ell_1 \gamma_2 + \ell_2 a_1^2 - \ell_1 a_1^1 +2 k_2 - \ell_2 + \ell_1 -2 a_2^1 \ell_1 \right)k_1     
      + k_1 + k_2 -2 k_1 k_2 -1  \\
     = \gamma_1 (-\ell_2 k_1 +k_2 a_1^1) + \gamma_2 (\ell_1 k_1 + k_2 a_1^2) + \ell_2 a_1^2 k_1 - \ell_1 a_1^1 k_1 - \ell_2 k_1 + \ell_1 k_1  - 2 a_2^1 \ell_1 k_1 -2a_1^1 a_1^2 k_2 + a_1^1 k_2 +a_1^2 k_2 -2a_1^2 a_2^1 k_2 + k_1 -1 \\
     = a_2^1 \gamma_1 + a_2^2 \gamma_2  +\ell_2 a_1^2 k_1 - {\ell_1 a_1^1 k_1} - {\ell_2 k_1} + {\ell_1 k_1} - {2a_2^1 \ell_1 k_1} - {2a_1^1 a_1^2 k_2} + {a_1^1 k_2} + {a_1^2 k_2} - {2a_1^2 a_2^1 k_2} + k_1 -1 \\
    = a_2^1 \gamma_1 + a_2^2 \gamma_2  + {\ell_2 a_1^2 k_1} - {\ell_1 a_1^1 k_1} + {a_2^1} + {a_2^2} {- 2a_2^1 a_2^2} - {2a_1^1 a_1^2 k_2} + k_1 -1 \\
    = a_2^1 \gamma_1 + a_2^2 \gamma_2 - a_1^2 a_2^1 - a_1^1 a_2^2 +a_2^1 +a_2^2 -2 a_2^1 a_2^2 + k_1 -1 \\
    = a_2^1 \gamma_1 + a_2^2 \gamma_2 -2a_1^2 a_2^1 +a_2^1 +a_2^2 -2 a_2^1 a_2^2 -1.
\end{dmath}


Plugging \eqref{beta11}-\eqref{eta} in \eqref{nbeta1} and \eqref{nbeta2} and using the fact that $k_1 = \det(A)$ we get
\begin{eqnarray}\nonumber
0 \leq a_1^1 \gamma_1 + a_1^2 \gamma_2 -2a_1^1 a_1^2 +a_1^1 +a_1^2 -2a_1^2 a_2^1 -1 \leq 2k_1-2,\\ \nonumber
0 \leq a_2^1 \gamma_1 + a_2^2 \gamma_2 -2a_1^2 a_2^1 +a_2^1 +a_2^2 -2 a_2^1 a_2^2 -1 \leq 2k_1 -2.
\end{eqnarray}
i.e.
\begin{eqnarray}\label{mn1}
m_1 \leq a_1^1 \gamma_1 + a_1^2 \gamma_2  \leq n_1,\\ \label{mn2}
m_2 \leq a_2^1 \gamma_1 + a_2^2 \gamma_2  \leq n_2.
\end{eqnarray}
where
\begin{eqnarray}\label{mnmn}
\begin{aligned}
m_1 &=& 2a_1^1 a_1^2 - a_1^1 - a_1^2 + 2a_1^2 a_2^1 +1, \quad m_2 = 2 a_1^2 a_2^1 - a_2^1 - a_2^2 + 2 a_2^1 a_2^2 +1,\\
n_1 &=& 2a_1^1 a_2^2 -a_1^1 -a_1^2 +2a_1^1 a_1^2 -1, \quad n_2 = 2a_1^1 a_2^2 -a_2^1 -a_2^2 + 2 a_2^1 a_2^2 -1.
\end{aligned}
\end{eqnarray}
Multiply \eqref{mn1} by the negative integer $(-a_2^2)$ and switching the upper and lower bounds then adding to \eqref{mn2} multiplied by the positive integer $(a_1^2)$ we get
\begin{equation}\label{jan}
m_2 a_1^2 -n_1 a_2^2 \leq -\det(A)\cdot \gamma_1 \leq n_2 a_1^2 -m_1 a_2^2.
\end{equation}
Dividing \eqref{jan} by the negative quantity $-\det(A)$, we get
\begin{equation}\label{jan1}
\dfrac{n_2 a_1^2 -m_1 a_2^2}{-\det(A)} \leq  \gamma_1 \leq \dfrac{m_2 a_1^2 -n_1 a_2^2}{-\det(A)}.
\end{equation}
Similarly, we multiply \eqref{mn1} by the negative integer $(-a_2^1)$ and switching the upper and lower bounds then adding to \eqref{mn2} multiplied by the positive integer $(a_1^1)$, we get
\begin{equation}\label{jan2}
m_2 a_1^1 -n_1 a_2^1 \leq \det(A)\cdot \gamma_2 \leq n_2 a_1^1 - m_1 a_2^1.
\end{equation}
Dividing \eqref{jan2} by $\det(A)>0$, we get
\begin{equation}\label{jan3}
\dfrac{m_2 a_1^1 -n_1 a_2^1}{\det(A)} \leq  \gamma_2 \leq \dfrac{n_2 a_1^1 - m_1 a_2^1}{\det(A)}.
\end{equation}
Using \eqref{mnmn}, we evaluate the lower and the upper bounds of $\gamma_1$ and $\gamma_2$ in \eqref{jan1} and \eqref{jan3} to get
\begin{eqnarray}\label{jan4}
\dfrac{n_2 a_1^2 -m_1 a_2^2}{-\det(A)} &=& -1 + \dfrac{a_1^2 +a_2^2}{\det(A)} > -1,\\ \label{jan5}
\dfrac{m_2 a_1^2 -n_1 a_2^2}{-\det(A)} &=& 2a_1^2 +2a_2^2 -1 - \dfrac{a_1^2 + a_2^2}{\det(A)} < 2a_1^2 +2a_2^2 -1,\\ \label{jan6}
\dfrac{m_2 a_1^1 -n_1 a_2^1}{\det(A)} &=& -1 + \dfrac{a_1^1 + a_2^1}{\det(A)} > -1,\\ \label{jan7}
\dfrac{n_2 a_1^1 -m_1 a_2^1}{\det(A)} &=& 2 a_1^1 + 2 a_2^1  - 1 - \dfrac{a_1^1 + a_2^1}{\det(A)} < 2 a_1^1 + 2 a_2^1  - 1.
\end{eqnarray}
Plugging \eqref{jan4}-\eqref{jan5} in \eqref{jan1}, and \eqref{jan6}-\eqref{jan7} in \eqref{jan3} and using the fact that $\gamma_1$ and $ \gamma_2$ are integers,  we get
\begin{eqnarray}\nonumber
0 \leq \gamma_1 \leq 2a_1^2 +2a_2^2 -2, \\ \nonumber
0 \leq \gamma_2 \leq 2 a_1^1 + 2 a_2^1  - 2.
\end{eqnarray}

This shows that the range of the values of $\gamma_1$ and $\gamma_2$ are nonnegatives and the sum in \eqref{bergU33} is finite, and proved \eqref{range1}.

Plugging \eqref{eta} in \eqref{bergU1} and since $k_1=\det(A)$, we get

\begin{equation}\nonumber
K_{\mathcal{U}}(p,q)= \dfrac{1}{\pi^2 \det(A)}\cdot \dfrac{\mathlarger{\sum\limits_{\gamma \in \mathbb{N}^2}}D_{\det(A)}(\nu)D_{\det(A)}(\eta)t_1^{\gamma_1} t_2^{\gamma_2}}{ \left(t_2^{a_2^1}-t_1^{a_2^2}\right)^2 \left(t_1^{a_1^2}-t_2^{a_1^1} \right)^2},
\end{equation}
where
\begin{eqnarray} \nonumber
\nu &=& a_1^1 \gamma_1 + a_1^2 \gamma_2 -2a_1^1 a_1^2 +a_1^1 +a_1^2 -2a_1^2 a_2^1-1,\\ \nonumber
\eta &=& a_2^1 \gamma_1 + a_2^2 \gamma_2 -2a_1^2 a_2^1 +a_2^1 +a_2^2 -2 a_2^1 a_2^2 -1.
\end{eqnarray}
By definition of $N(\cdot ,\cdot )$, defined in \eqref{funnn}, on the matrix $A$
\begin{eqnarray}\nonumber
N(A,\gamma) &=& a_1^1 \gamma_1 + a_1^2 \gamma_2 - 2a_1^1 a_1^2 + a_1^1 +a_1^2 -1 -\perm(A) + \abs{\det(A)}.
\end{eqnarray}
Since $\det(A) >0$, we have
\begin{eqnarray}\nonumber
N(A,\gamma) &=& a_1^1 \gamma_1 + a_1^2 \gamma_2 - 2a_1^1 a_1^2 + a_1^1 +a_1^2 -1 -\left(a_1^1 a_2^2 +a_1^2 a_2^1 \right) + \left(a_1^1 a_2^2 - a_1^2 a_2^1 \right) \\ \nonumber
&=& a_1^1 \gamma_1 + a_1^2 \gamma_2 - 2a_1^1 a_1^2 + a_1^1 +a_1^2 -2a_1^2 a_2^1-1.
\end{eqnarray}
It follows that $\nu = N(A,\gamma).$ Similarly,
\begin{eqnarray} \nonumber
N(AR,\gamma) &=& a_2^1 \gamma_1 + a_2^2 \gamma_2 -2a_2^1 a_2^2 +a_2^1 +a_2^2 -1 - \perm(AR) +\abs{\det(AR)}.\\ \nonumber
&=& a_2^1 \gamma_1 + a_2^2 \gamma_2 -2a_2^1 a_2^2 +a_2^1 +a_2^2 -1 - \left( a_2^1 a_1^2 +a_1^1 a_2^2 \right) + \abs{ a_2^1 a_1^2 - a_1^1 a_2^2 } \\ \nonumber
&=& a_2^1 \gamma_1 + a_2^2 \gamma_2 -2a_2^1 a_2^2 +a_2^1 +a_2^2 -1 - \left( a_2^1 a_1^2 +a_1^1 a_2^2 \right) + \abs{-\det(A) }\\ \nonumber
&=& a_2^1 \gamma_1 + a_2^2 \gamma_2 -2a_2^1 a_2^2 +a_2^1 +a_2^2 -1 - \left( a_2^1 a_1^2 +a_1^1 a_2^2 \right) + \det(A) \\ \nonumber
&=& a_2^1 \gamma_1 + a_2^2 \gamma_2 -2a_2^1 a_2^2 +a_2^1 +a_2^2 -1 - \left( a_2^1 a_1^2 +a_1^1 a_2^2 \right) + a_1^1 a_2^2 - a_1^2 a_2^1 \\ \nonumber
&=& a_2^1 \gamma_1 + a_2^2 \gamma_2 -2a_2^1 a_2^2 +a_2^1 +a_2^2 -1  -2 a_1^2 a_2^1 \\ \nonumber
&=& a_2^1 \gamma_1 + a_2^2 \gamma_2 -2a_1^2 a_2^1 +a_2^1 +a_2^2 -2 a_2^1 a_2^2 -1.
\end{eqnarray} 
It follows that $N(AR,\gamma)= \eta.$ We can see the full formula
\begin{equation}\nonumber
K_{\mathcal{U}}(z,w)= \dfrac{1}{\pi^2 \det(A)}\cdot \dfrac{g(t_1,t_2)}{ \left(t_2^{a_2^1}-t_1^{a_2^2}\right)^2 \left(t_1^{a_1^2}-t_2^{a_1^1} \right)^2}
\end{equation}
with
\begin{equation}\nonumber
g(t_1,t_2) = \mathlarger{\sum\limits_{\gamma \in \mathbb{N}^2}}D_{\det(A)}(N(A,\gamma))D_{\det(A)}(N(AR,\gamma))t_1^{\gamma_1} t_2^{\gamma_2}.
\end{equation}
This completes the proof of Theorem \ref{mainthm1} and Theorem \ref{mainthm2}.

\bibliographystyle{plain}
\bibliography{mybib}

\begin{thebibliography}{10}

\bibitem{bender}
Chase Bender, Debraj Chakrabarti, Luke Edholm, and Meera Mainkar.
\newblock {$L^p$}-regularity of the {B}ergman projection on quotient domains.
\newblock {\em Canad. J. Math.}, 74(3):732--772, 2022.

\bibitem{bremer}
H.~J. Bremermann.
\newblock Holomorphic continuation of the kernel function and the {B}ergman
  metric in several complex variables.
\newblock In {\em Lectures on functions of a complex variable}, pages 349--383.
  University of Michigan Press, Ann Arbor, Mich., 1955.

\bibitem{austin}
Debraj Chakrabarti, Austin Konkel, Meera Mainkar, and Evan Miller.
\newblock Bergman kernels of elementary {R}einhardt domains.
\newblock {\em Pacific J. Math.}, 306(1):67--93, 2020.

\bibitem{chen}
Liwei Chen.
\newblock The {$L^p$} boundedness of the {B}ergman projection for a class of
  bounded {H}artogs domains.
\newblock {\em J. Math. Anal. Appl.}, 448(1):598--610, 2017.

\bibitem{luke1}
Luke~D. Edholm.
\newblock Bergman theory of certain generalized {H}artogs triangles.
\newblock {\em Pacific J. Math.}, 284(2):327--342, 2016.

\bibitem{edho}
Luke~David Edholm.
\newblock {\em The {B}ergman kernel of fat {H}artogs triangles}.
\newblock ProQuest LLC, Ann Arbor, MI, 2016.
\newblock Thesis (Ph.D.)--The Ohio State University.

\bibitem{hua}
L.~K. Hua.
\newblock {\em Harmonic analysis of functions of several complex variables in
  the classical domains}.
\newblock American Mathematical Society, Providence, R.I., 1963.
\newblock Translated from the Russian by Leo Ebner and Adam Kor\'{a}nyi.

\bibitem{mass}
William~S. Massey.
\newblock {\em A basic course in algebraic topology}, volume 127 of {\em
  Graduate Texts in Mathematics}.
\newblock Springer-Verlag, New York, 1991.

\bibitem{new}
Morris Newman.
\newblock {\em Integral matrices}.
\newblock Pure and Applied Mathematics, Vol. 45. Academic Press, New
  York-London, 1972.

\bibitem{range}
R.~Michael Range.
\newblock {\em Holomorphic functions and integral representations in several
  complex variables}, volume 108 of {\em Graduate Texts in Mathematics}.
\newblock Springer-Verlag, New York, 1986.

\bibitem{zhang}
Shuo Zhang.
\newblock {$L^p$} {S}obolev mapping properties of the {B}ergman projections on
  {$n$}-dimensional generalized {H}artogs triangles.
\newblock {\em Bull. Korean Math. Soc.}, 58(6):1355--1375, 2021.

\end{thebibliography}

\end{document}